\documentclass[reqno]{amsart}
\usepackage{amssymb}
\usepackage{amscd}
\usepackage{color}

\newtheorem{theorem}{Theorem}[section]
\newtheorem{lemma}[theorem]{Lemma}
\newtheorem{question}[theorem]{Question}
\newtheorem{definition}[theorem]{Definition}
\newtheorem{example}[theorem]{Example}
\newtheorem{remark}[theorem]{Remark}
\newtheorem{proposition}[theorem]{Proposition}
\newtheorem{corollary}[theorem]{Corollary}

\begin{document}
\title[Unique Factorization Property II]
{Unique Factorization property of non-Unique Factorization Domains II}
\author[G. W. Chang]{Gyu Whan Chang}
\address{Department of Mathematics Education, Incheon National University, Incheon 22012, Korea}
\email{whan@inu.ac.kr}

\author[A. Reinhart]{Andreas Reinhart}
\address{Institut f\"ur Mathematik und wissenschaftliches Rechnen, Karl-Franzens-Universit\"at Graz, NAWI Graz, Heinrichstra{\ss}e 36, 8010 Graz, Austria}
\email{andreas.reinhart@uni-graz.at}

\subjclass[2010]{13A15, 13F05, 13G05}
\keywords{valuation element, VFD, P$v$MD, HoFD, weakly Matlis GCD-domain}
\date{\today}

\begin{abstract}
Let $D$ be an integral domain. A nonzero nonunit $a$ of $D$ is called a {\em valuation element} if there is a valuation overring $V$ of $D$ such that $aV\cap D=aD$. We say that $D$ is a {\em valuation factorization domain} (VFD) if each nonzero nonunit of $D$ can be written as a finite product of valuation elements. In this paper, we study some ring-theoretic properties of VFDs. Among other things, we show that (i) a VFD $D$ is Schreier, and hence ${\rm Cl}_t(D)=\{0\}$, (ii) if $D$ is a P$v$MD, then $D$ is a VFD if and only if $D$ is a weakly Matlis GCD-domain, if and only if $D[X]$, the polynomial ring over $D$, is a VFD and (iii) a VFD $D$ is a weakly factorial GCD-domain if and only if $D$ is archimedean. We also study a unique factorization property of VFDs.
\end{abstract}

\maketitle

\section*{0. Introduction}

Let $D$ be an integral domain with quotient field $K$. An overring of $D$ means a subring of $K$ containing $D$. A nonzero nonunit $x\in D$ is said to be {\em homogeneous} if $x$ is contained in a unique maximal $t$-ideal of $D$. As in \cite{c19}, we say that $D$ is a {\em homogeneous factorization domain} (HoFD) if each nonzero nonunit of $D$ can be written as a finite product of homogeneous elements. Let $D$ be an HoFD, let $x\in D$ be a nonzero nonunit and let $x=\prod_{i=1}^n a_i=\prod_{j=1}^m b_j$ be two finite products of $t$-comaximal homogeneous elements of $D$. Then $n=m$ and $a_iD=b_iD$ for $i\in [1,n]$ by reordering if necessary \cite[Remark 2.1]{c19}. Hence, an HoFD has a unique factorization property even though it is not a unique factorization domain (UFD). In \cite{c19}, Chang studied several properties of HoFDs and constructed examples of HoFDs. In this paper, we continue to study the unique factorization property of non-unique factorization domains.

As in \cite[Appendix 3]{zs61}, we say that an ideal $I$ of $D$ is a {\em valuation ideal} if there is a valuation overring $V$ of $D$ such that $IV\cap D=I$. Clearly, each ideal of a valuation domain is a valuation ideal. Conversely, in \cite[Corollary 2.4]{go65}, Gilmer and Ohm showed that if every principal ideal of $D$ is a valuation ideal, then $D$ is a valuation domain. In this paper, we will say that a nonzero nonunit $a\in D$ is a {\em valuation element} if $aD$ is a valuation ideal, i.e., there is a valuation overring $V$ of $D$ such that $aV\cap D=aD$. It is well known that a prime ideal of $D$ is a valuation ideal \cite[page 341]{zs61}. Hence, every prime element is a valuation element. Thus, every nonzero nonunit of a UFD can be written as a finite product of valuation elements. We will say that $D$ is a {\em valuation factorization domain} (VFD) if each nonzero nonunit of $D$ can be written as a finite product of valuation elements. Clearly, valuation domains and UFDs are VFDs. The purpose of this paper is to study some factorization properties of VFDs.

\subsection{Definitions related to the $t$-operation.}
We first review some definitions related to the $t$-operation which are needed for fully understanding this paper. Let $D$ be an integral domain with quotient field $K$. A $D$-submodule $A$ of $K$ is called a fractional ideal of $D$ if $dA\subseteq D$ for some nonzero $d\in D$. Let $F(D)$ (resp., $f(D)$) be the set of nonzero fractional (resp., nonzero finitely generated fractional) ideals of $D$. For $A\in F(D)$, let $A^{-1}=\{x\in K\mid xA\subseteq D\}$; then $A^{-1}\in F(D)$. Hence, if we set

\vspace{.1cm}
\begin{itemize}
\item $A_v=(A^{-1})^{-1}$ and

\item $A_t=\bigcup\{I_v\mid I\subseteq A$ and $I\in f(D)\}$,
\end{itemize}

\vspace{.1cm}
\noindent
then the $v$- and $t$-operations are well defined. It is easy to see that $I\subseteq I_t\subseteq I_v$ for all $I\in F(D)$ and $I_t=I_v$ if $I$ is finitely generated. Let $*=v$ or $t$. An $I\in F(D)$ is called a {\em $*$-ideal} if $I_*=I$. A $*$-ideal is a {\em maximal $*$-ideal} if it is maximal among the proper integral $*$-ideals. Let $*$-${\rm Max}(D)$ be the set of maximal $*$-ideals of $D$. It may happen that $v$-${\rm Max}(D)=\emptyset$ even though $D$ is not a field as in the case of a rank-one nondiscrete valuation domain $D$. However, $t$-${\rm Max}(D)\neq\emptyset$ if and only if $D$ is not a field; each maximal $t$-ideal of $D$ is a prime ideal; each proper $t$-ideal of $D$ is contained in a maximal $t$-ideal; each prime ideal of $D$ minimal over a $t$-ideal is a $t$-ideal, whence each height-one prime ideal is a $t$-ideal; and $D=\bigcap_{P\in t\text{-{\rm Max}}(D)}D_P$. An integral domain $D$ is said to be of {\em finite $(t$-$)$character} if each nonzero nonunit of $D$ is contained in only finitely many maximal ($t$-)ideals. Let ${\rm Spec}(D)$ (resp., $t$-${\rm Spec}(D)$) be the set of prime ideals (resp., prime $t$-ideals) of $D$; so $t$-Max$(D) \subseteq t$-Spec$(D) \subseteq$ Spec$(D) \setminus \{(0)\}$. The $t$-dimension of $D$ is defined by $t$-$\dim(D)=\sup\{n\mid P_1\subsetneq\cdots\subsetneq P_n$ for some prime $t$-ideals $P_i$ of $D\}$. Hence, $t$-$\dim(D)=1$ if and only if $D$ is not a field and $t$-${\rm Max}(D)=t$-${\rm Spec}(D)$, and if $\dim(D)=1$, then $t$-${\rm Max}(D)=t$-${\rm Spec}(D)={\rm Spec}(D)\setminus \{(0)\}$.

An $I\in F(D)$ is said to be {\em invertible} (resp., {\em $t$-invertible}) if $II^{-1}=D$ (resp., $(II^{-1})_t=D$). It is easy to see that invertible ideals are $t$-invertible $t$-ideals. We say that $D$ is a {\em Pr\"ufer $v$-multiplication domain} (P$v$MD) if each nonzero finitely generated ideal of $D$ is $t$-invertible. It is known that $D$ is a P$v$MD if and only if $D_P$ is a valuation domain for all maximal $t$-ideals $P$ of $D$, if and only if $D[X]$, the polynomial ring over $D$, is a P$v$MD \cite[Theorems 3.2 and 3.7]{k89}; and a Pr\"ufer domain is a P$v$MD whose maximal ideals are $t$-ideals. Let ${\rm T}(D)$ be the set of $t$-invertible fractional $t$-ideals. Then ${\rm T}(D)$ is an abelian group under $I*J=(IJ)_t$. Let ${\rm Inv}(D)$ (resp., ${\rm Prin}(D)$) be the subgroup of ${\rm T}(D)$ of invertible (resp., nonzero principal) fractional ideals of $D$. The factor group ${\rm Cl}_t(D)={\rm T}(D)/{\rm Prin}(D)$, called the {\em $t$-class group} of $D$, is an abelian group and ${\rm Pic}(D)={\rm Inv}(D)/{\rm Prin}(D)$, called the Picard group of $D$, is a subgroup of ${\rm Cl}_t(D)$. A GCD-domain is just a P$v$MD with trivial $t$-class group.

\subsection{Results}
This paper consists of five sections including the introduction. Let $D$ be an integral domain. In Section 1, we study basic properties of valuation elements and VFDs. Among other things, we show that (i) a VFD is integrally closed, (ii) if $D$ is not a field, then $D$ is a VFD with $t$-$\dim(D)=1$ if and only if $D$ is a weakly factorial GCD-domain and (iii) every nonzero nonunit of a VFD can be written as a finite product of incomparable valuation elements. In Section 2, we show that (i) a VFD is a Schreier domain, and hence it has a trivial $t$-class group and (ii) a UMT-domain $D$ is a VFD if and only if $D[X]$ is a VFD. In Section 3, we study VFDs that are HoFDs. We show that if $t$-${\rm Spec}(D)$ is treed, then (i) every valuation element is a homogeneous element and (ii) $D$ is a VFD if and only if $D$ is a weakly Matlis GCD-domain, if and only if $D[X]$ is a VFD. Finally, in Section 4, we introduce the notion of UVFDs and show that the UVFDs are precisely the weakly Matlis GCD-domains. We also characterize when a VFD is a UVFD.

\section{Valuation elements and VFDs}

Let $D$ be an integral domain with quotient field $K$. Let $\mathbb{N}$ be the set of positive integers and let $\mathbb{N}_0$ be the set of non-negative integers. For elements $a,b\in D$, we say that $a$ divides $b$ (denoted by $a\mid_D b$) if $b=ac$ for some $c\in D$. In this section, we study basic properties of valuation elements and VFDs. Our first result is very simple, but it plays a key role in the study of VFDs.

\begin{proposition}\label{proposition1.1}
Let $D$ be an integral domain, let $D'$ be an overring of $D$ and let $a,b\in D$ be such that $a\not=0$ and $aD'\cap D=aD$.
\begin{enumerate}
\item If $bD'\cap D=bD$, then $abD'\cap D=abD$.
\item If $b\mid_D a$, then $bD'\cap D=bD$.
\item If $\sqrt{aD}\subseteq\sqrt{bD}$, then $bD'\cap D=bD$. In particular, if $a$ is a valuation element of $D$ and $\sqrt{aD}\subseteq\sqrt{bD}\subsetneq D$, then $b$ is a valuation element of $D$.
\end{enumerate}
\end{proposition}

\begin{proof}
(1) Let $bD'\cap D=bD$. Observe that $abD=a(bD)=a(bD'\cap D)=abD'\cap aD=abD'\cap aD'\cap D=abD'\cap D$.

(2) Let $b\mid_D a$. There is some $c\in D$ such that $a=bc$. We infer that $bD'\cap D=c^{-1}aD'\cap D=c^{-1}(aD'\cap cD)=c^{-1}(aD'\cap D\cap cD)=c^{-1}(aD\cap cD)=bD\cap D=bD$.

(3) Let $\sqrt{aD}\subseteq\sqrt{bD}$. Then there is some $k\in\mathbb{N}$ such that $a^k\in bD$.
By (1), $a^kD'\cap D=a^kD$, and since $b\mid_D a^k$, we infer by (2) that $bD'\cap D=bD$.
\end{proof}

\begin{corollary}\label{corollary1.2}
Let $D$ be an integral domain and let $a\in D$ be a valuation element.
\begin{enumerate}
\item If $b\in D$ is such that $\sqrt{aD}\subseteq\sqrt{bD}$, then $aD$ and $bD$ are comparable.
\item Each two principal ideals of $D$ that contain $a$ are comparable.
\item $\bigcap_{n\in\mathbb{N}} a^nD\in {\rm Spec}(D)$.
\end{enumerate}
\end{corollary}

\begin{proof}
There is some valuation overring $V$ of $D$ such that $aV\cap D=aD$.

(1) Let $b\in D$ be such that $\sqrt{aD}\subseteq\sqrt{bD}$. It follows from Proposition~\ref{proposition1.1}(3) that $bV\cap D=bD$. Since $V$ is a valuation domain, we have that $aV$ and $bV$ are comparable, and hence $aD$ and $bD$ are comparable.

(2) Let $b,c\in D$ be such that $a\in bD\cap cD$. Then $bV\cap D=bD$ and $cV\cap D=cD$ by Proposition~\ref{proposition1.1}(2). Since $bV$ and $cV$ are comparable, we infer that $bD$ and $cD$ are comparable.

(3) It follows from Proposition~\ref{proposition1.1}(1) that $a^nV\cap D=a^nD$ for each $n\in\mathbb{N}$. Therefore, $(\bigcap_{n\in\mathbb{N}} a^nV)\cap D=\bigcap_{n\in\mathbb{N}} a^nD$. Since $a$ is not a unit of $V$, we have that $\bigcap_{n\in\mathbb{N}} a^nV\in {\rm Spec}(V)$, and thus $\bigcap_{n\in\mathbb{N}} a^nD\in {\rm Spec}(D)$.
\end{proof}

\begin{remark}\label{remark1.3}
{\em Let $D$ be a VFD, let $a\in D$ be a valuation element and let $Q\in {\rm Spec}(D)$ be such that $Q\subsetneq\sqrt{aD}$. Then $Q\subseteq\bigcap_{n\in\mathbb{N}} a^nD$.}
\end{remark}

\begin{proof}
Let $x\in Q\setminus\{0\}$. Then $x\in bD\subseteq Q$ for some valuation element $b\in D$. We have that $\sqrt{bD}\subsetneq\sqrt{aD}=\sqrt{a^nD}$ for each $n\in\mathbb{N}$, and hence $x\in bD\subseteq\bigcap_{n\in\mathbb{N}} a^nD$ by Corollary~\ref{corollary1.2}(1). Consequently, $Q\subseteq\bigcap_{n\in\mathbb{N}} a^nD$.
\end{proof}

\begin{corollary}\label{corollary1.4}\cite[Corollary 2.4]{go65}
Let $D$ be an integral domain. If every nonzero nonunit of $D$ is a valuation element, then $D$ is a valuation domain.
\end{corollary}

\begin{proof}
Let $a,b\in D$ be nonzero nonunits. Then $ab$ is a valuation element by assumption. Note that $aD$ and $bD$ are principal ideals of $D$ that contain $ab$. Consequently, $aD$ and $bD$ are comparable by Corollary~\ref{corollary1.2}(2). Therefore, $D$ is a valuation domain.
\end{proof}

\begin{corollary}\label{corollary1.5}
A VFD is integrally closed.
\end{corollary}

\begin{proof}
Let $D$ be a VFD, let $\overline{D}$ be the integral closure of $D$ and let $a\in D$ be a valuation element. There is some valuation overring $V$ of $D$ such that $aV\cap D=aD$. Since $V$ is integrally closed, it follows that $\overline{D}\subseteq V$, and hence $a\overline{D}\cap D=aD$. It is an immediate consequence of Proposition~\ref{proposition1.1}(1) that $x\overline{D}\cap D=xD$ for each $x\in D$. If $y\in\overline{D}$, then $yz\in D$ for some nonzero $z\in D$, and thus $yz\in z\overline{D}\cap D=zD$ and $y\in D$. Consequently, $D=\overline{D}$.
\end{proof}

\begin{corollary}\label{corollary1.6}
Let $D$ be a quasi-local domain of dimension one. The following statements are equivalent.
\begin{enumerate}
\item $D$ is a valuation domain.
\item $D$ has at least one valuation element.
\item $D$ is a VFD.
\end{enumerate}
\end{corollary}

\begin{proof}
(1) $\Rightarrow$ (3) $\Rightarrow$ (2) This is clear.

(2) $\Rightarrow$ (1) Let $a\in D$ be a valuation element and let $b\in D$ be a nonzero nonunit. Then $\sqrt{aD}=\sqrt{bD}$, and hence $b$ is a valuation element of $D$ by Proposition~\ref{proposition1.1}(3). Therefore, $D$ is a valuation domain by Corollary~\ref{corollary1.4}.
\end{proof}

A nonzero nonunit $x$ of $D$ is said to be {\em primary} if $xD$ is a primary ideal. Clearly, prime elements are primary but not vice versa.
\begin{proposition}\label{proposition1.7}
Let $D$ be an integral domain, let $a\in D$ be a valuation element and let $S$ be a multiplicatively closed subset of $D$.
\begin{enumerate}
\item $\sqrt{aD}$ is a prime $t$-ideal.
\item $a$ is a primary element if and only if $\sqrt{aD}$ is a maximal $t$-ideal.
\item If $t$-$\dim(D)=1$, then every valuation element of $D$ is a primary element.
\item If $aS^{-1}D\subsetneq S^{-1}D$, then $a$ is a valuation element of $S^{-1}D$.
\end{enumerate}
\end{proposition}

\begin{proof}
(1) Let $V$ be a valuation overring of $D$ such that $aD=aV\cap D$. Then $\sqrt{aD}=\sqrt{aV}\cap D$, and since $\sqrt{aV}$ is a prime ideal, $\sqrt{aD}$ is a prime ideal. Clearly, $\sqrt{aD}$ is minimal over $aD$ and $aD$ is a $t$-ideal. Thus, $\sqrt{aD}$ is a prime $t$-ideal.

(2) This follows from \cite[Lemma 2.1]{acp03}.

(3) Let $t$-$\dim(D)=1$ and let $b\in D$ be a valuation element. Then $\sqrt{bD}$ is a maximal $t$-ideal by (1) and assumption. Thus, by (2), $b$ is a primary element.

(4) Let $aS^{-1}D\subsetneq S^{-1}D$. There is some valuation overring $V$ of $D$ such that $aV\cap D=aD$. Observe that $S$ is a multiplicatively closed subset of $V$, and hence $S^{-1}V$ is an overring of $V$. Since $V$ is a valuation domain, we have that $S^{-1}V$ is a valuation domain. Note that $aS^{-1}D=S^{-1}(aD)=S^{-1}(aV\cap D)=aS^{-1}V\cap S^{-1}D$. Thus, $a$ is a valuation element of $S^{-1}D$.
\end{proof}

\begin{corollary}\label{corollary1.8}
Let $D$ be a VFD and let $S$ be a multiplicatively closed subset of $D$.
\begin{enumerate}
\item $S^{-1}D$ is a VFD.
\item If $P$ is a height-one prime ideal of $D$, then $D_P$ is a valuation domain.
\end{enumerate}
\end{corollary}

\begin{proof}
(1) This follows directly from Proposition~\ref{proposition1.7}(4).

(2) This is an immediate consequence of (1) and Corollary~\ref{corollary1.6}.
\end{proof}

An integral domain $D$ is a {\em weakly factorial domain} (WFD) if every nonzero nonunit of $D$ can be written as a finite product of primary elements. Let $X^1(D)$ be the set of height-one prime ideals of $D$. It is known that $D$ is a WFD if and only if $D=\bigcap_{P\in X^1(D)} D_P$, where the intersection is locally finite (i.e., for each nonzero $x\in D$, $x$ is a unit of $D_P$ for all but finitely many $P\in X^1(D)$) and ${\rm Cl}_t(D)=\{0\}$ \cite[Theorem]{az90}; in this case, $t$-$\dim(D)=1$ (cf. Proposition~\ref{proposition1.7}(2)).

\begin{corollary}\label{corollary1.9}
Let $D$ be an integral domain that is not a field. Then $D$ is a VFD with $t$-$\dim(D)=1$ if and only if $D$ is a weakly factorial GCD-domain.
\end{corollary}

\begin{proof}
$(\Rightarrow)$ Let $D$ be a VFD of $t$-dimension one. It follows from Proposition~\ref{proposition1.7}(3) that $D$ is a weakly factorial domain. Thus, it is an immediate consequence of Corollary~\ref{corollary1.8}(2) and \cite[Theorem 18]{am88} that $D$ is a GCD-domain.

$(\Leftarrow)$ Now let $D$ be a weakly factorial GCD-domain. Then $t$-$\dim(D)=1$. We next show that every primary element is a valuation element. Let $a\in D$ be a primary element and let $P=\sqrt{aD}$. Then $P$ is a height-one prime ideal of $D$, and since $D$ is a GCD-domain, $D_P$ is a valuation domain. Note that $aD_P\cap D=aD$, and hence $a$ is a valuation element. Thus, $D$ is a VFD.
\end{proof}

Note that if $D$ is a (one-dimensional) B\'ezout domain which is not of finite character (e.g., let $D$ be the example in \cite[Theorem 3.4]{v78} or let $D$ be the ring of entire functions), then $D$ is a GCD-domain and yet $D$ is not a VFD (by Theorem~\ref{theorem3.4}). For more details concerning this example, we refer to \cite[Example 4.2]{r12}.

For $n\in\mathbb{N}$, let $[1,n]=\{k\in\mathbb{N}\mid 1\leq k\leq n\}$. Two elements $x$ and $y$ of an integral domain $D$ are said to be {\it incomparable} if $xD$ and $yD$ are incomparable under inclusion. We next show that each nonzero nonunit $a$ of a VFD $D$ can be written as a finite product of incomparable valuation elements, say, $a=\prod_{i=1}^n a_i$, and in this case, $n$ is the number of minimal prime ideals of $aD$ by a series of lemmas.

\begin{lemma}\label{lemma1.10}
Let $D$ be an integral domain. If $v\in D$ is a finite product of valuation elements of $D$ such that $\sqrt{vD}$ is a prime ideal, then $v$ is a valuation element.
\end{lemma}

\begin{proof}
Let $k\in\mathbb{N}$ and let $v\in D$ be such that $\sqrt{vD}$ is a prime ideal of $D$ and $v=\prod_{i=1}^k v_i$ for some valuation elements $v_i\in D$. We have that $\sqrt{vD}=\bigcap_{i=1}^k\sqrt{v_iD}$. Since $\sqrt{vD}$ is a prime ideal of $D$, it follows that $\sqrt{vD}=\sqrt{v_jD}$ for some $j \in [1,k]$. It is an immediate consequence of Proposition~\ref{proposition1.1}(3) that $v$ is a valuation element of $D$.
\end{proof}

\begin{corollary}\label{corollary1.11}
Let $D$ be a VFD. Then the valuation elements of $D$ are precisely the elements $a\in D$ for which $\sqrt{aD}$ is a nonzero prime ideal of $D$.
\end{corollary}

\begin{proof}
This is an immediate consequence of Proposition~\ref{proposition1.7}(1) and Lemma~\ref{lemma1.10}.
\end{proof}

Let $I$ be an ideal of an integral domain $D$. Let $\mathcal{P}(I)$ denote the set of minimal prime ideals of $I$.

\begin{lemma}\label{lemma1.12}
Let $D$ be a VFD and let $a\in D$ be a nonzero nonunit. Then
\begin{center}
$\min\{k\in\mathbb{N}\mid a$ is a product of $k$ valuation elements of $D\}=|\mathcal{P}(aD)|$.
\end{center}
\end{lemma}

\begin{proof}
Let $n=\min\{k\in\mathbb{N}\mid a$ is a product of $k$ valuation elements of $D\}$. We have that $a=\prod_{i=1}^n a_i$ for some valuation elements $a_i$ of $D$. Let $P\in\mathcal{P}(aD)$. Set $\Sigma_P=\{i\in [1,n]\mid P\subseteq\sqrt{a_iD}\}$. Observe that $\sqrt{a_jD}=P$ for some $j\in\Sigma_P$. This implies that $\sqrt{\prod_{i\in\Sigma_P} a_iD}=\bigcap_{i\in\Sigma_P}\sqrt{a_iD}=P$, and thus $\prod_{i\in\Sigma_P} a_i$ is a valuation element of $D$ by Corollary~\ref{corollary1.11}. Since $n$ is minimal and $\Sigma_P$ is nonempty, we infer that $|\Sigma_P|=1$. Note that $[1,n]=\bigcup_{Q\in\mathcal{P}(aD)}\Sigma_Q$. Consequently, there is a bijection $\varphi:\mathcal{P}(aD)\rightarrow [1,n]$ such that $\Sigma_Q=\{\varphi(Q)\}$ for each $Q\in\mathcal{P}(aD)$.
\end{proof}

\begin{proposition}\label{proposition1.13}
Let $D$ be a VFD, let $n\in\mathbb{N}$, let $(a_i)_{i=1}^n$ be a sequence of valuation elements of $D$ and let $a=\prod_{i=1}^n a_i$. The following statements are equivalent.
\begin{enumerate}
\item $a_i$ and $a_j$ are incomparable for all distinct $i,j\in [1,n]$.
\item $\sqrt{a_iD}$ and $\sqrt{a_jD}$ are incomparable for all distinct $i,j\in [1,n]$.
\item A map $f:[1,n]\rightarrow\mathcal{P}(aD)$ given by $f(i)=\sqrt{a_iD}$ is a well-defined bijection.
\item $n=|\mathcal{P}(aD)|$.
\end{enumerate}
Hence, every nonzero nonunit of $D$ can be written as a finite product of incomparable valuation elements.
\end{proposition}

\begin{proof}
(1) $\Rightarrow$ (2) This follows from Corollary~\ref{corollary1.2}(1).

(2) $\Rightarrow$ (3) Note that if $P\in\mathcal{P}(aD)$, then $a_i\in P$ for some $i\in [1,n]$, and hence $P=\sqrt{a_iD}$. Moreover, if $j\in [1,n]$, then $a\in\sqrt{a_jD}$, and thus $Q\subseteq\sqrt{a_jD}$ for some $Q\in\mathcal{P}(aD)$. As shown before, $Q=\sqrt{a_kD}$ for some $k\in [1,n]$. It follows that $k=j$, and hence $\sqrt{a_jD}=Q\in\mathcal{P}(aD)$. Thus, $f$ is a well-defined bijection.

(3) $\Rightarrow$ (4) This is obvious.

(4) $\Rightarrow$ (1) If there are distinct $i,j\in [1,n]$ such that $a_iD$ and $a_jD$ are comparable, then $a_ia_j$ is a valuation element, which contradicts Lemma~\ref{lemma1.12}.

\vspace{.12cm}
Moreover, by Lemma~\ref{lemma1.12} again, every nonzero nonunit of $D$ can be written as a finite product of incomparable valuation elements.
\end{proof}

Now let $D$ be a VFD. It is an easy consequence of Proposition~\ref{proposition1.13} that if $n,m\in\mathbb{N}$ and $(a_i)_{i=1}^n$ and $(b_j)_{j=1}^m$ are two sequences of incomparable valuation elements of $D$ with $\prod_{i=1}^n a_i=\prod_{j=1}^m b_j$, then $n=m$ and $\sqrt{a_iD}=\sqrt{b_iD}$ for each $i\in [1,n]$ by reordering if necessary.

\begin{corollary}\label{corollary1.14}
Let $D$ be a VFD and let $\Omega=\{\sqrt{xD}\mid x\in D\setminus\{0\},\sqrt{xD}\in {\rm Spec}(D)\}$.
\begin{enumerate}
\item The valuation elements of $D$ are precisely the nonzero nonunits $a\in D$ for which each two principal ideals of $D$ that contain $a$ are comparable.
\item If $a\in D$ is a valuation element and $P,Q\in\Omega$ are such that $a\in P\cap Q$, then $P$ and $Q$ are comparable.
\item $\Omega=\bigcup_{a\in D\setminus\{0\}}\mathcal{P}(aD)=\{\sqrt{xD}\mid x\in D$ is a valuation element$\}$.
\end{enumerate}
\end{corollary}

\begin{proof}
(1) This is an easy consequence of Corollary~\ref{corollary1.2}(2) and Proposition~\ref{proposition1.13}.

(2) Let $a\in D$ be a valuation element and let $P,Q\in\Omega$ be such that $a\in P\cap Q$. Then $\sqrt{aD}\in\Omega$ and $\sqrt{aD}\subseteq P\cap Q$. Moreover $P=\sqrt{pD}$ and $Q=\sqrt{qD}$ for some $p,q\in D$. Without restriction let $\sqrt{aD}\subsetneq P$ and $\sqrt{aD}\subsetneq Q$. Therefore, $aD\subseteq pD\cap qD$ by Corollary~\ref{corollary1.2}(1), and thus $pD$ and $qD$ are comparable by (1). Consequently, $P$ and $Q$ are comparable.

(3) ($\subseteq$) First let $P\in\Omega$. There is some nonzero $x\in D$ such that $P=\sqrt{xD}$. Observe that $P\in\mathcal{P}(xD)$. ($\subseteq$) Next let $a\in D$ be a nonzero nonunit and let $Q\in\mathcal{P}(aD)$. Then $a\in yD\subseteq Q$ for some valuation element $y\in D$. It follows from Corollary~\ref{corollary1.11} that $\sqrt{yD}$ is a prime ideal of $D$, and hence $Q=\sqrt{yD}$. ($\subseteq$) Finally, let $z\in D$ be a valuation element of $D$. Set $A=\sqrt{zD}$. It follows from Corollary~\ref{corollary1.11} that $A\in\Omega$.
\end{proof}

\section{Schreier domains}

Let $D$ be an integral domain. Then $D$ is called a {\em pre-Schreier domain} if for all nonzero $x,y,z\in D$ with $x\mid_D yz$, there are some $a,b\in D$ such that $x=ab$, $a\mid_D y$ and $b\mid_D z$. Moreover, $D$ is called a {\em Schreier domain} if $D$ is an integrally closed pre-Schreier domain. Clearly, GCD-domains are Schreier domains. Schreier domains were introduced by Cohn \cite{co68}, and later, in \cite{z87}, Zafrullah introduced the notion of pre-Schreier domains.

(Pre-)Schreier domains are rather ``nice'' integral domains. Let $D[X]$ be the polynomial ring over $D$. Recall that a polynomial $f\in D[X]$ is called {\em primitive} if each common divisor of the coefficients of $f$ is a unit of $D$. We say that $D$ satisfies {\it Gau\ss' Lemma} if the product of each two primitive polynomials over $D$ is primitive. Clearly, UFDs satisfy Gau\ss' Lemma, and we use this fact to show that if $D$ is a UFD, then $D[X]$ is also a UFD. It is well known (cf. \cite[Propositions 3.2 and 3.3]{az07}) that every (pre-)Schreier domain satisfies Gau\ss' Lemma.

\begin{proposition}\label{proposition2.1}
A VFD is a Schreier domain.
\end{proposition}

\begin{proof}
Let $D$ be a VFD. It follows from Corollary~\ref{corollary1.5} that $D$ is integrally closed. Next we show by induction that for each $k\in\mathbb{N}$, for each valuation element $x\in D$ and for all nonzero $y,z\in D$ such that $x\mid_D yz$ and $|\mathcal{P}(yD)|+|\mathcal{P}(zD)|=k$, there are some $a,b\in D$ such that $x=ab$, $a\mid_D y$ and $b\mid_D z$.

Let $k\in\mathbb{N}$, let $x\in D$ be a valuation element and let $y,z\in D$ be nonzero such that $x\mid_D yz$ and $|\mathcal{P}(yD)|+|\mathcal{P}(zD)|=k$. Observe that $yz\in xD\subseteq\sqrt{xD}\in {\rm Spec}(D)$, and hence $y\in\sqrt{xD}$ or $z\in\sqrt{xD}$. Without restriction let $y\in\sqrt{xD}$. Note that $y=\prod_{P\in\mathcal{P}(yD)} y_P$ for some incomparable valuation elements $y_P\in D$ with $\sqrt{y_PD}=P$ for each $P\in\mathcal{P}(yD)$ by Proposition~\ref{proposition1.13}. Consequently, there is some $P\in\mathcal{P}(yD)$ such that $\sqrt{y_PD}\subseteq\sqrt{xD}$. It follows from Corollary~\ref{corollary1.2}(1) that $y_PD$ and $xD$ are comparable.

\vspace{.2cm}
{\noindent}{\textsc{Case 1}}: $y_PD\subseteq xD$. Then $y\in xD$. Set $a=x$ and $b=1$. Then $x=ab$, $a\mid_D y$ and $b\mid_D z$.

\vspace{.2cm}
{\noindent}{\textsc{Case 2}}: $xD\subsetneq y_PD$. There is some nonunit $w\in D$ such that $x=y_Pw$. We infer by Proposition~\ref{proposition1.1}(2) that $w$ is a valuation element of $D$. There is some $y^{\prime}\in D$ such that $y=y_Py^{\prime}$. Note that $\mathcal{P}(y^{\prime}D)=\mathcal{P}(yD)\setminus\{P\}$, and hence $|\mathcal{P}(y^{\prime}D)|+|\mathcal{P}(zD)|<k$. Moreover, $y_Pw=x\mid_D yz=y_Py^{\prime}z$, and thus $w\mid_D y^{\prime}z$. It follows by the induction hypothesis that $w=a^{\prime}b$, $a^{\prime}\mid_D y^{\prime}$ and $b\mid_D z$ for some $a^{\prime},b\in D$. Set $a=y_Pa^{\prime}$. Then $x=ab$ and $a\mid_D y$.

\vspace{.2cm}
We infer that for each valuation element $x\in D$ and for all nonzero $y,z\in D$ with $x\mid_D yz$ there are some $a,b\in D$ such that $x=ab$, $a\mid_D y$ and $b\mid_D z$. Now it is straightforward to show by induction that for each $n\in\mathbb{N}$, for each $x\in D$ which is a product of $n$ valuation elements of $D$ and for each nonzero $y,z\in D$ with $x\mid_D yz$, there are some $a,b\in D$ such that $x=ab$, $a\mid_D y$ and $b\mid_Dz$. This implies that $D$ is a Schreier domain.
\end{proof}

However, Schreier domains need not be VFDs. For example, it is known that a Pr\"ufer domain is a Schreier domain if and only if it is a B\'ezout domain \cite[Proposition 2]{dz11}, while a VFD that is a Pr\"ufer domain is an h-local Pr\"ufer domain (by Corollary~\ref{corollary3.7}). Hence, if $D=\mathbb{Z}_{\mathbb{Z}\setminus (2)\cup (3)}+X\mathbb{Q}[\![X]\!]$, then $D$ is a Schreier domain but not a VFD.

\begin{corollary}\label{corollary2.2}
Let $D$ be an integral domain. The following statements are equivalent.
\begin{enumerate}
\item $D$ is a VFD.
\item $D$ is a Schreier domain and every nonzero prime $t$-ideal of $D$ contains a valuation element of $D$.
\item $D$ is a pre-Schreier domain and every nonzero prime $t$-ideal of $D$ contains a valuation element of $D$.
\end{enumerate}
\end{corollary}

\begin{proof}
(1) $\Rightarrow$ (2) It follows from Proposition~\ref{proposition2.1} that $D$ is a Schreier domain. It is obvious that every nonzero prime $t$-ideal of $D$ contains a valuation element of $D$.

(2) $\Rightarrow$ (3) This is obvious.

(3) $\Rightarrow$ (1) Let $\Sigma$ be the set of finite products of units and valuation elements of $D$. Observe that $\Sigma$ is a multiplicatively closed subset of $D$. Next we show that $\Sigma$ is divisor-closed. Let $x\in D$ be such that $x\mid_D y$ for some $y\in\Sigma$. There are some $n\in\mathbb{N}$ and some elements $y_i\in D$ for each $i\in [1,n]$ which are either units or valuation elements of $D$ such that $y=\prod_{i=1}^n y_i$. Therefore, $x=\prod_{i=1}^n x_i$ for some elements $x_i\in D$ such that $x_i\mid_D y_i$ for each $i\in [1,n]$. It follows from Proposition~\ref{proposition1.1}(2) that $x_i$ is a unit or a valuation element of $D$ for each $i\in [1,n]$. Therefore, $x\in\Sigma$.

It is sufficient to show that $D\setminus\{0\}\subseteq\Sigma$. Assume that there is some $z\in D\setminus (\Sigma\cup\{0\})$. Then $zD\cap\Sigma=\emptyset$, because $\Sigma$ is divisor-closed by the previous paragraph. Consequently, there is some prime $t$-ideal $P$ of $D$ such that $zD\subseteq P$ and $P\cap\Sigma=\emptyset$. On the other hand, $P$ contains a valuation element of $D$, and hence $P\cap\Sigma\not=\emptyset$, a contradiction.
\end{proof}

Let $I$ be a $t$-ideal of an integral domain $D$. Then $I$ is said to be {\it $t$-finite} if $I=J_t$ for some $J\in f(D)$. It is known that $I$ is $t$-invertible if and only if $I$ is $t$-finite and $I_P$ is principal for all $P\in t$-Max$(D)$ \cite[Corollary 2.7]{k89}.

\begin{corollary}\label{corollary2.3}
Let $D$ be a VFD.
\begin{enumerate}
\item ${\rm Cl}_t(D)=\{0\}$.
\item Every atom of $D$ is a prime element.
\item If $D$ is a $t$-finite conductor domain, i.e., the intersection of each two principal ideals of $D$ is $t$-finite,
 then $D$ is a GCD-domain.
\end{enumerate}
\end{corollary}

\begin{proof}
This follows directly from Proposition~\ref{proposition2.1} and \cite[Proposition 2 and Corollary 6]{dz11}.
\end{proof}

A {\it Mori domain} is an integral domain which satisfies the ascending chain condition on integral $t$-ideals, equivalently, each $t$-ideal is of finite type. Mori domains include Noetherian domains, Krull domains and UFDs.

\begin{corollary}\label{corollary2.4}
Let $D$ be a VFD. The following statements are equivalent.
\begin{enumerate}
\item $D$ is atomic.
\item $D$ is a Mori domain.
\item $D$ is a Krull domain.
\item $D$ is a UFD.
\end{enumerate}
\end{corollary}

\begin{proof}
Clearly, every UFD is a Krull domain and every Krull domain is a Mori domain. By definition, a Mori domain satisfies the ascending chain condition on principal ideals and is, therefore, atomic. If $D$ is atomic, then every nonzero nonunit of $D$ is a finite product of prime elements by Corollary~\ref{corollary2.3}(2), and thus $D$ is a UFD.
\end{proof}

Observe that if $D$ is a Krull domain which is not a UFD (e.g., $D=\mathbb{Z}[\sqrt{-5}]$), then $D$ and $D[X]$ are both examples of Krull domains that fail to be VFDs. Furthermore, if $V$ is a nondiscrete valuation domain (e.g., $V=\overline{\mathbb{Z}}_M$, where $\overline{\mathbb{Z}}$ is the ring of algebraic integers and $M$ is a maximal ideal of $\overline{\mathbb{Z}}$), then $V$ is a VFD and yet $V$ is not atomic. We next study when $D[X]$ is a VFD.

\begin{lemma}\label{lemma2.5}
Let $D[X]$ be the polynomial ring over an integral domain $D$ and let $a\in D$ be a nonzero nonunit. Then $a$ is a valuation element of $D$ if and only if $a$ is a valuation element of $D[X]$.
\end{lemma}

\begin{proof}
Let $K$ be the quotient field of $D$.

$(\Rightarrow)$ By assumption, $aV\cap D=aD$ for some valuation overring $V$ of $D$. Note that if $M$ is a maximal ideal of $V$,
then $V(X):=V[X]_{M[X]}$ is a valuation overring of $D[X]$ and $V(X)\cap K[X]=V[X]$; hence if $u\in aV(X)\cap D[X]$, then $u=af$ for some $f\in V[X]$. Hence, $aV\cap D=aD$ implies $af\in aD[X]$, and thus $f\in D[X]$. Therefore, $aV(X)\cap D[X]=aD[X]$.

$(\Leftarrow)$ Let $W$ be a valuation overring of $D[X]$ such that $aW\cap D[X]=aD[X]$ and let $V=W\cap K$. Then $V$ is a valuation overring of $D$ and
\begin{eqnarray*}
aD &=& aD[X]\cap K=aW\cap D[X]\cap K\\
&=& (aW\cap aK)\cap (D[X]\cap K)
=a(W\cap K)\cap D\\ &=& aV\cap D.
\end{eqnarray*}
Thus, $a$ is a valuation element of $D$.
\end{proof}

Let $D[X]$ be the polynomial ring over $D$. For $f\in D[X]$, let $c(f)$ denote the ideal of $D$ generated by the coefficients of $f$. A nonzero prime ideal $Q$ of $D[X]$ is called an {\em upper to zero} in $D[X]$ if $Q\cap D=(0)$. Following \cite{hz89}, we say that $D$ is a {\em UMT-domain} if each upper to zero in $D[X]$ is a maximal $t$-ideal. Then $D$ is an integrally closed UMT-domain if and only if $D$ is a P$v$MD \cite[Proposition 3.2]{hz89}.

\begin{proposition}\label{proposition2.6}
Let $D[X]$ be the polynomial ring over an integral domain $D$. Then $D[X]$ is a VFD if and only if $D$ is a VFD and every upper to zero in $D[X]$ contains a valuation element of $D[X]$.
\end{proposition}

\begin{proof}
$(\Rightarrow)$ Let $D[X]$ be a VFD. Let $a\in D$ be a nonzero nonunit. Then $a$ is a nonzero nonunit of $D[X]$, and hence $a$ is a finite product of valuation elements of $D[X]$. Clearly, each of these valuation elements is contained in $D$, and thus $a$ is a finite product of valuation elements of $D$ by Lemma~\ref{lemma2.5}. Therefore, $D$ is a VFD. It is clear that every upper to zero in $D[X]$ contains a valuation element of $D[X]$.

$(\Leftarrow)$ Let $D$ be a VFD and let every upper to zero in $D[X]$ contain a valuation element of $D[X]$. Then $D$ is a Schreier domain by Proposition~\ref{proposition2.1}. It follows from \cite[Theorem 2.7]{co68} (or from \cite[Theorem 4.8]{az07}) that $D[X]$ is a Schreier domain. Let $P$ be a nonzero prime $t$-ideal of $D[X]$. If $P\cap D=(0)$, then $P$ contains a valuation element of $D[X]$ by assumption. Now let $P\cap D\not=(0)$. It is clear that $P\cap D$ contains a valuation element of $D$. It follows from Lemma~\ref{lemma2.5} that $P$ contains a valuation element of $D[X]$. Consequently, $D[X]$ is a VFD by Corollary~\ref{corollary2.2}.
\end{proof}

\begin{corollary}\label{corollary2.7}
Let $D[X]$ be the polynomial ring over a UMT-domain $D$. Then $D$ is a VFD if and only if $D[X]$ is a VFD.
\end{corollary}

\begin{proof}
By Proposition~\ref{proposition2.6} it is sufficient to show that if $D$ is a VFD, then every upper to zero in $D[X]$ contains a valuation element of $D[X]$. Let $D$ be a VFD. Then $D$ is integrally closed by Corollary~\ref{corollary1.5}, and hence $D$ is a P$v$MD because $D$ is a UMT-domain. Hence, $D$ is a GCD-domain by Corollary~\ref{corollary2.3}(1). Let $P$ be an upper to zero in $D[X]$. Since $D$ is a GCD-domain, it follows that $P=fD[X]$ for some prime element $f\in D[X]$. Since every prime element of $D[X]$ is a valuation element, it follows that $P$ contains a valuation element of $D[X]$.
\end{proof}

\section{VFDs which are HoFDs}

An integral domain $D$ is called a {\em weakly Matlis domain} if (i) $D$ is of finite $t$-character and (ii) $D$ is independent, i.e., no two distinct maximal $t$-ideals of $D$ contain a common nonzero prime ideal. It is easy to see that if $D$ is not a field, then $D$ is a weakly factorial GCD-domain if and only if $D$ is a weakly Matlis GCD-domain with $t$-dim$(D)=1$.

Let $D$ be an integral domain. We say that $a,b\in D$ are {\rm $t$-comaximal} if $(a,b)_v=D$. Two elements $a,b\in D$ are said to be {\it coprime} if for each $c\in D$ with $aD\cup bD\subseteq cD$, it follows that $c$ is a unit of $D$. Hence, if $a,b \in D$ are $t$-comaximal, then $a,b$ are coprime. Note that if $a,b\in D$ are two homogeneous elements that are not $t$-comaximal, then $ab$ is also a homogeneous element of $D$. Thus, every nonzero nonunit of an HoFD can be written as a finite product of $t$-comaximal homogeneous elements. It is known that $D$ is an HoFD if and only if $D$ is a weakly Matlis domain with trivial $t$-class group \cite[Theorem 2.2]{c19}. We first study when VFDs are HoFDs.

\begin{proposition}\label{proposition3.1}
Let $D$ be a VFD. The following statements are equivalent.
\begin{enumerate}
\item $D$ is an HoFD.
\item $D$ is a weakly Matlis domain.
\item Every nonzero prime $t$-ideal of $D$ is contained in a unique maximal $t$-ideal.
\item Every valuation element of $D$ is homogeneous.
\item $D$ is of finite $t$-character.
\item Every valuation element of $D$ is contained in only finitely many maximal $t$-ideals.
\end{enumerate}
If every maximal $t$-ideal of $D$ is the radical of a principal ideal, then these equivalent conditions are satisfied.
\end{proposition}

\begin{proof}
(1) $\Rightarrow$ (2) This follows from \cite[Theorem 2.2]{c19}.

(2) $\Rightarrow$ (3) This is clear.

(3) $\Rightarrow$ (4) Let $a\in D$ be a valuation element of $D$. Then $\sqrt{aD}$ is a nonzero prime $t$-ideal of $D$ by Proposition~\ref{proposition1.7}(1). Consequently, $|\{M\in t$-${\rm Max}(D)\mid a\in M\}|=|\{M\in t$-${\rm Max}(D)\mid\sqrt{aD}\subseteq M\}|=1$, and hence $a$ is homogeneous.

(4) $\Rightarrow$ (1) This is obvious.

(2) $\Rightarrow$ (5) $\Rightarrow$ (6) This is clear.

(6) $\Rightarrow$ (4) Let $a\in D$ be a valuation element. Set $\Sigma=\{Q\in t$-${\rm Max}(D)\mid a\in Q\}$. Assume that $|\Sigma|\geq 2$. Then there are some distinct $M,N\in\Sigma$. Since $\Sigma$ is finite, there are some $b\in M\setminus\bigcup_{Q\in\Sigma\setminus\{M\}} Q$ and $c\in N\setminus\bigcup_{Q\in\Sigma\setminus\{N\}} Q$. Note that $a$ and $b$ are not $t$-comaximal. We infer by Proposition~\ref{proposition2.1} and \cite[Proposition 3.3]{az07} that $aD\cup bD\subseteq dD$ for some nonunit $d\in D$. It follows by analogy that $aD\cup cD\subseteq eD$ for some nonunit $e\in D$. Since $a\in dD\cap eD$, we have that $dD$ and $eD$ are comparable by Corollary~\ref{corollary1.2}(2). Without restriction let $eD\subseteq dD$. There is some $A\in t$-${\rm Max}(D)$ such that $dD\subseteq A$, and hence $aD\cup bD\cup cD\subseteq A$. This implies that $M=A=N$, a contradiction.

Now let every maximal $t$-ideal of $D$ be the radical of a principal ideal. We infer by Corollary~\ref{corollary1.14}(2) that every valuation element of $D$ is homogeneous.
\end{proof}

We say that $D$ is a {\em $t$-treed domain} if the set of prime $t$-ideals of $D$ is treed under inclusion. The class of $t$-treed domains includes P$v$MDs and integral domains of $t$-dimension one. We next study VFDs that are $t$-treed. We first need a lemma.

\begin{lemma}\label{lemma3.2}
Let $D$ be a $t$-treed domain. Then every valuation element of $D$ is homogeneous.
\end{lemma}

\begin{proof}
Let $a\in D$ be a valuation element. Assume to the contrary that $a$ is not homogeneous. Hence, there are at least two distinct maximal $t$-ideals $M_1$ and $M_2$ of $D$ containing $a$. Let $S=D\setminus (M_1\cup M_2)$, and note that $a\in S^{-1}D$ is a valuation element by Proposition~\ref{proposition1.7}(4). Hence, by replacing $D$ with $S^{-1}D$, we assume that $D$ is a treed domain with two maximal ideals $M_1$ and $M_2$.

Since $a$ is a valuation element, there is a valuation overring $V$ of $D$ such that $aV\cap D=aD$. Note that if $M$ is the maximal ideal of $V$, then $M\cap D$ is a proper prime ideal of $D$, and hence, without loss of generality, we may assume that $M\cap D\subseteq M_1$. Thus, $aD\subseteq aD_{M_1}\cap D\subseteq aV\cap D=aD$, whence $aD_{M_1}\cap D=aD$. Choose $b\in M_2\setminus M_1$. Then $\sqrt{aD}\subsetneq\sqrt{bD}$ because ${\rm Spec}(D)$ is treed. Thus, by Proposition~\ref{proposition1.1}(3), $bD=bD_{M_1}\cap D=D_{M_1}\cap D=D$, a contradiction. Therefore, $a$ is contained in a unique maximal $t$-ideal of $D$.
\end{proof}

Next we want to point out that a weaker form of Lemma~\ref{lemma3.2} can be proved without relying on prime avoidance.

\begin{remark}\label{remark3.3}
{\em Let $D$ be a VFD. If each two valuation elements of $D$ that are incomparable are $t$-comaximal, then every valuation element of $D$ is homogeneous.}
\end{remark}

\begin{proof}
Let $a\in D$ be a valuation element. Assume to the contrary that $a$ is not homogeneous. Then there are distinct maximal $t$-ideals $M$ and $Q$ of $D$ containing $a$. Since $D$ is a VFD, there are valuation elements $b,c\in D$ such that $b\in M\setminus Q$ and $c\in Q\setminus M$. Since $aD$ and $bD$ are contained in $M$, we infer that $aD$ and $bD$ are comparable, and hence $aD\subseteq bD$. It follows by analogy that $aD\subseteq cD$. Therefore, $bD$ and $cD$ are comparable by Corollary~\ref{corollary1.2}(2), a contradiction.
\end{proof}

A P$v$MD is a ring of Krull type if it is of finite $t$-character, and a P$v$MD is an independent ring of Krull type if it is weakly Matlis. Recall that a P$v$MD $D$ is an HoFD if and only if $D[X]$ is an HoFD, if and only if $D$ is an independent ring of Krull type with ${\rm Cl}_t(D)=\{0\}$ \cite[Corollary 2.6]{c19}.

\begin{theorem}\label{theorem3.4}
The following statements are equivalent for a $t$-treed domain $D$.
\begin{enumerate}
\item $D$ is a VFD.
\item $D$ is an HoFD and a P$v$MD.
\item $D$ is an independent ring of Krull type and ${\rm Cl}_t(D)=\{0\}$.
\item $D[X]$ is a VFD.
\item $D$ is a weakly Matlis GCD-domain.
\end{enumerate}
\end{theorem}

\begin{proof}
(1) $\Rightarrow$ (2) By Lemma~\ref{lemma3.2}, it suffices to show that $D$ is a P$v$MD. Let $M$ be a maximal $t$-ideal of $D$. Then $D_M$ is a VFD by Corollary~\ref{corollary1.8}(1) and ${\rm Spec}(D_M)$ is linearly ordered by assumption. Let $a\in D_M$ be a nonzero nonunit. Then, since ${\rm Spec}(D_M)$ is linearly ordered, $\sqrt{aD_M}$ is a prime ideal, and hence $a$ is a valuation element of $D_M$ by Corollary~\ref{corollary1.11}. Thus, $D_M$ is a valuation domain by Corollary~\ref{corollary1.4}.

(2) $\Rightarrow$ (1) Let $a\in D$ be a homogeneous element. Then there is a unique maximal $t$-ideal $M$ of $D$ with $a\in M$. Hence, $$aD=\bigcap_{P\in t\text{-{\rm Max}}(D)}aD_P=aD_M\cap\Big(\bigcap_{P\in t\text{-{\rm Max}}(D)\setminus\{M\}}D_P\Big)=aD_M\cap D,$$ and since $D_M$ is a valuation domain, $a$ is a valuation element of $D$. Thus, $D$ is a VFD.

(2) $\Leftrightarrow$ (3) This follows from \cite[Corollary 2.6]{c19}.

(1) $\Rightarrow$ (4) If $D$ is a VFD, then $D$ is a P$v$MD by the proof of (1) $\Rightarrow$ (2). Thus, $D[X]$ is a VFD by Corollary~\ref{corollary2.7}.

(4) $\Rightarrow$ (1) This is an immediate consequence of Proposition~\ref{proposition2.6}.

(3) $\Leftrightarrow$ (5) This follows because a GCD-domain is a P$v$MD with trivial $t$-class group.
\end{proof}

Next we want to point out that even a Schreier domain with a unique maximal $t$-ideal need not be a VFD. In particular, weakly Matlis Schreier domains and Schreier domains which are HoFDs need not be VFDs. Recall that a quasi-local integral domain $D$ with maximal ideal $M$ is a {\it pseudo-valuation domain \textnormal{(}PVD\textnormal{)}} if for all ideals $A$ and $B$ of $D$, it follows that $A\subseteq B$ or $BM\subseteq AM$
\cite[Theorem 1.4]{hh78}.

\begin{example}\label{example3.5}\cite[Example 2.10]{az07}
{\em Let $T=\mathbb{C}[X]$, let $K$ be a quotient field of $T$ and let $S$ be the integral closure of $T$ in an algebraic closure $\overline{K}$ of $K$. Let $Q$ be a maximal ideal of $S$, let $\overline{\mathbb{Q}}\subseteq\overline{K}$ be the algebraic closure of $\mathbb{Q}$ and let $D=\overline{\mathbb{Q}}+Q_Q$. Then $D$ is a Schreier domain with a unique maximal $t$-ideal and yet $D$ is not a VFD.}
\end{example}

\begin{proof}
It follows from \cite[Example 2.10]{az07} and its proof that $D$ is a Schreier domain and a PVD, but not a B\'ezout domain. Since $D$ is a PVD, we have that ${\rm Spec}(D)$ is linearly ordered, and thus $D$ is $t$-treed. In particular, $D$ has a unique maximal $t$-ideal. Since $D$ is not a B\'ezout domain, it follows that $D$ is not a valuation domain, and thus $D$ is not a GCD-domain. Therefore, $D$ is not a VFD by Theorem~\ref{theorem3.4}.
\end{proof}

\begin{corollary}\label{corollary3.6}
A P$v$MD $D$ is a VFD if and only if $D$ is an HoFD.
\end{corollary}

\begin{proof}
It is well known that a P$v$MD is a $t$-treed domain. Thus, the result follows directly from Theorem~\ref{theorem3.4}.
\end{proof}

In \cite[Section 4]{c19}, Chang studied HoFDs that are P$v$MDs and he also constructed several examples of such kind of integral domains. An integral domain $D$ has {\it finite character} if each nonzero element of $D$ is contained in at most finitely many maximal ideals of $D$. The domain $D$ is said to be {\it h-local} if $D$ has finite character and each nonzero prime ideal of $D$ is contained in a unique maximal ideal. Hence, $D$ is an h-local domain if $D$ is a weakly Matlis domain whose maximal ideals are $t$-ideals.

\begin{corollary}\label{corollary3.7}
A Pr\"ufer domain $D$ is a VFD if and only if $D$ is an h-local Pr\"ufer domain with ${\rm Pic}(D)=\{0\}$.
\end{corollary}

\begin{proof}
It is clear that a Pr\"ufer domain is an independent ring of Krull type if and only if it is an h-local Pr\"ufer domain. Thus, the result follows directly from Theorem~\ref{theorem3.4}.
\end{proof}

Let $D$ be a UMT-domain or a $t$-treed domain. Then $D$ is a VFD if and only if $D[X]$, the polynomial ring over $D$, is a VFD
by Corollary~\ref{corollary2.7} and Theorem~\ref{theorem3.4}. However, we don't know if this is true in general.

\begin{question}\label{question3.8}
{\em Let $D[X]$ be the polynomial ring over a VFD D. Is $D[X]$ a VFD?}
\end{question}

\section{Unique valuation factorization domains}

For $n\in\mathbb{N}$ let $\mathcal{S}_n$ be the symmetric group on $n$ letters.

\begin{definition}\label{definition4.1}
{\em Let $D$ be an integral domain. We say that $D$ is a {\it unique VFD \textnormal{(}UVFD\textnormal{)}} if the following two conditions are satisfied.
\begin{enumerate}
\item Every nonzero nonunit of $D$ is a finite product of incomparable valuation elements of $D$.
\item If $n,m\in\mathbb{N}$ and $(a_i)_{i=1}^n$ and $(b_j)_{j=1}^m$ are two sequences of incomparable valuation elements of $D$ with $\prod_{i=1}^n a_i=\prod_{j=1}^m b_j$, then $n=m$ and there is some $\sigma\in\mathcal{S}_n$ such that $a_iD=b_{\sigma(i)}D$ for each $i\in [1,n]$.
\end{enumerate}}
\end{definition}

Clearly, UVFDs are VFDs by Proposition~\ref{proposition1.13}. Moreover, by the remark after Proposition~\ref{proposition1.13}, it follows that a VFD $D$ is a UVFD if and only if for all $n\in\mathbb{N}$ and all sequences $(a_i)_{i=1}^n$ and $(b_i)_{i=1}^n$ of incomparable valuation elements of $D$ with $\prod_{i=1}^n a_i=\prod_{i=1}^n b_i$ and $\sqrt{a_jD}=\sqrt{b_jD}$ for each $j\in [1,n]$, it follows that $a_jD=b_jD$ for each $j\in [1,n]$.

\begin{theorem}\label{theorem4.2}
Let $D$ be a VFD and let $\Omega=\{\sqrt{xD}\mid x\in D\setminus\{0\},\sqrt{xD}\in {\rm Spec}(D)\}$. The following statements are equivalent.
\begin{enumerate}
\item $D$ is a UVFD.
\item Each two incomparable valuation elements of $D$ are coprime.
\item $D$ is a P$v$MD.
\item $D_P$ is a valuation domain for each $P\in\Omega$.
\item For all $A,B,C\in\Omega$ with $A\cup B\subseteq C$, $A$ and $B$ are comparable.
\end{enumerate}
\end{theorem}

\begin{proof}
(1) $\Rightarrow$ (2) Let $a,b\in D$ be incomparable valuation elements of $D$ and let $c\in D$ be such that $aD\cup bD\subseteq cD$. Set $v=ac$ and $w=bc$. Then $\sqrt{vD}=\sqrt{aD}$, $\sqrt{wD}=\sqrt{bD}$ and $vb=aw$. It follows from Corollary~\ref{corollary1.11} that $v$ and $w$ are valuation elements of $D$. Note that $\sqrt{aD}$ and $\sqrt{bD}$ are incomparable by Corollary~\ref{corollary1.2}(1). Therefore, $v$ and $b$ are incomparable and $a$ and $w$ are incomparable. We infer that $vD=aD$ by assumption, and hence $c$ is a unit of $D$.

(2) $\Rightarrow$ (3) By Theorem~\ref{theorem3.4}, it is sufficient to show that $D$ is $t$-treed. Assume to the contrary that there is some $M\in t$-${\rm Max}(D)$ and some incomparable prime $t$-ideals $P$ and $Q$ of $D$ that are contained in $M$. Observe that there are some valuation elements $a,b\in D$ such that $a\in P\setminus Q$ and $b\in Q\setminus P$. It is clear that $a$ and $b$ are incomparable. Therefore, $a$ and $b$ are coprime. It follows from Proposition~\ref{proposition2.1} and \cite[Proposition 3.3]{az07} that $a$ and $b$ are $t$-comaximal, a contradiction.

(3) $\Rightarrow$ (4) This is clear.

(4) $\Rightarrow$ (5) Let $A,B,P\in\Omega$ be such that $A\cup B\subseteq P$. Then $D_P$ is a valuation domain and $A_P$ and $B_P$ are prime ideals of $D_P$. Hence, $A_P$ and $B_P$ are comparable, and thus $A=A_P\cap D$ and $B=B_P\cap D$ are comparable.

(5) $\Rightarrow$ (1) Let $n\in\mathbb{N}$ and let $(a_i)_{i=1}^n$ and $(b_i)_{i=1}^n$ be two sequences of incomparable valuation elements of $D$ such that $\prod_{i=1}^n a_i=\prod_{i=1}^n b_i$ and $\sqrt{a_iD}=\sqrt{b_iD}$ for each $i\in [1,n]$. Let $i\in [1,n]$. Since $a_i\mid_D\prod_{j=1}^n b_j$ we infer by Proposition~\ref{proposition2.1} that $a_i=\prod_{j=1}^n b^{\prime}_j$ for some elements $b^{\prime}_j$ of $D$ such that $b^{\prime}_j\mid_D b_j$ for each $j\in [1,n]$. If $j\in [1,n]\setminus\{i\}$, then $\sqrt{a_iD}$ and $\sqrt{b_jD}$ are incomparable, and since $\sqrt{a_iD}\cup\sqrt{b_jD}\subseteq\sqrt{b^{\prime}_jD}$, we infer that $b^{\prime}_j$ is a unit of $D$. This implies that $a_iD=b^{\prime}_iD\supseteq b_iD$. It follows by analogy that $b_iD\supseteq a_iD$, and hence $a_iD=b_iD$. Thus, $D$ is a UVFD.
\end{proof}

\begin{corollary}\label{corollary4.3}
Let $D$ be a UVFD and let $S$ be a multiplicatively closed subset of $D$. Then $S^{-1}D$ is a UVFD.
\end{corollary}

\begin{proof}
It follows from Theorem~\ref{theorem4.2} that $D$ is a VFD and a P$v$MD. It is clear that $S^{-1}D$ is a P$v$MD. Moreover, $S^{-1}D$ is a VFD by Corollary~\ref{corollary1.8}(1). Therefore, $S^{-1}D$ is a UFVD by Theorem~\ref{theorem4.2}.
\end{proof}

\begin{corollary}\label{corollary4.4}
Let $D[X]$ be the polynomial ring over an integral domain $D$. Then $D$ is a UVFD if and only if $D[X]$ is a UVFD.
\end{corollary}

\begin{proof}
Note that $D$ is a P$v$MD if and only if $D[X]$ is a P$v$MD \cite[Theorem 3.7]{k89}. Thus, the result is an immediate consequence of Corollary~\ref{corollary2.7} and Theorem~\ref{theorem4.2}.
\end{proof}

\begin{corollary}\label{corollary4.5}
An integral domain is a UVFD if and only if it is a weakly Matlis GCD-domain.
\end{corollary}

\begin{proof}
This is an immediate consequence of Theorems~\ref{theorem3.4} and~\ref{theorem4.2}.
\end{proof}

Let $D$ be an integral domain with quotient field $K$. We say that $D$ satisfies the {\it Principal Ideal Theorem} if each minimal prime ideal of each nonzero principal ideal of $D$ is of height one. It is well known that Noetherian domains satisfy the Principal Ideal Theorem. Recall that an element $x\in K$ is said to be {\it almost integral} over $D$ if there exists some nonzero $c\in D$ such that $cx^n\in D$ for each $n\in\mathbb{N}$. We say that $D$ is {\it completely integrally closed} if for all $x\in K$ such that $x$ is almost integral over $D$, it follows that $x\in D$. Moreover, $D$ is said to be {\it archimedean} if for each nonunit $x\in D$, $\bigcap_{n\in\mathbb{N}} x^nD=(0)$. Observe that if $t$-$\dim(D)=1$, then $D$ satisfies the Principal Ideal Theorem. Moreover, if $D$ is completely integrally closed or $D$ satisfies the Principal Ideal Theorem, then $D$ is archimedean.

\begin{proposition}\label{proposition4.6}
Let $D$ be a VFD that is not a field. The following statements are equivalent.
\begin{enumerate}
\item $D$ is a weakly factorial GCD-domain.
\item $t$-$\dim(D)=1$.
\item $D$ is completely integrally closed.
\item $D$ is archimedean.
\item $D$ satisfies the Principal Ideal Theorem.
\end{enumerate}
\end{proposition}

\begin{proof}
Set $\Omega=\{\sqrt{xD}\mid x\in D\setminus\{0\},\sqrt{xD}\in {\rm Spec}(D)\}$. By Corollary~\ref{corollary1.14}(3), we have that $\Omega=\bigcup_{a\in D\setminus\{0\}}\mathcal{P}(aD)=\{\sqrt{xD}\mid x\in D$ is a valuation element$\}$.

(1) $\Leftrightarrow$ (2) This follows from Corollary~\ref{corollary1.9}.

(1) $\Rightarrow$ (3) Note that $D$ is an intersection of one-dimensional valuation overrings of $D$, and hence $D$ is an intersection of completely integrally closed overrings of $D$. Therefore, $D$ is completely integrally closed.

(3) $\Rightarrow$ (4) This is clear.

(4) $\Rightarrow$ (5) Let $P\in\Omega$. Then $P=\sqrt{pD}$ for some valuation element $p\in D$. It remains to show that $P$ is of height one. Let $Q$ be a prime ideal of $D$ such that $Q\subsetneq P$. We infer by Remark~\ref{remark1.3} that $Q\subseteq\bigcap_{n\in\mathbb{N}} p^nD=(0)$, and thus $Q=(0)$.

(5) $\Rightarrow$ (2) Note that $\Omega$ is the set of height-one prime ideals of $D$. It remains to show that each maximal $t$-ideal of $D$ is an element of $\Omega$. It follows by Corollary~\ref{corollary1.8}(2) that $D_P$ is a valuation domain for each $P\in\Omega$. Thus, $D$ is a P$v$MD by Theorem~\ref{theorem4.2}. Let $M\in t$-${\rm Max}(D)$. Then $M=\bigcup_{P\in\Omega,P\subseteq M} P$. Observe that $D$ is $t$-treed, and hence $M\in\Omega$.
\end{proof}

Note that if $V$ is a two-dimensional valuation domain (e.g., $V={\rm Int}(\mathbb{Z})_M$, where ${\rm Int}(\mathbb{Z})=\{f\in\mathbb{Q}[X]\mid f(\mathbb{Z})\subseteq\mathbb{Z}\}$ is the ring of integer-valued polynomials and $M$ is a height-two prime ideal of ${\rm Int}(\mathbb{Z})$), then $V$ and $V[X]$ are both examples of non-archimedean VFDs.

Let $D$ be an integral domain. We say that a nonzero element $a\in D$ has {\it prime radical} if $\sqrt{aD}$ is a prime ideal of $D$. Next we study VFDs in which each minimal prime ideal of a nonzero $t$-finite $t$-ideal is minimal over a $t$-invertible $t$-ideal (i.e., for each nonzero $t$-finite $t$-ideal $I$ of $D$ and every $P\in\mathcal{P}(I)$ there is some $t$-invertible $t$-ideal $J$ of $D$ such that $P\in\mathcal{P}(J)$). In other words, we study VFDs $D$ for which
\begin{eqnarray*}
&\bigcup&\{\mathcal{P}(I)\mid I\textnormal{ is a nonzero $t$-finite $t$-ideal of } D\}\\
=&\bigcup&\{\mathcal{P}(I)\mid I\textnormal{ is a $t$-invertible $t$-ideal of } D\}.
\end{eqnarray*}
Suppose that $D$ satisfies one of the following conditions.
\begin{enumerate}
\item $D$ is a P$v$MD.
\item $D$ is of $t$-dimension one.
\item $D$ has finitely many prime ideals.
\item For each $t$-finite $t$-ideal $I$ of $D$ there is some $a\in D$ such that $\sqrt{I}=\sqrt{aD}$.
\end{enumerate}
Then every minimal prime ideal of a nonzero $t$-finite $t$-ideal of $D$ is minimal over a $t$-invertible $t$-ideal of $D$.

\begin{lemma}\label{lemma4.7}
Let $D$ be an integral domain in which every nonzero nonunit is a finite product of elements with prime radical and let $I$ be a $t$-invertible $t$-ideal of $D$. Then $\mathcal{P}(I)$ is finite and for each $\emptyset\not=\mathcal{L}\subseteq\mathcal{P}(I)$, $\bigcap_{Q\in\mathcal{L}} Q$ is the radical of a principal ideal of $D$.
\end{lemma}

\begin{proof}
First we show that every prime ideal in $\mathcal{P}(I)$ is the radical of a principal ideal. Let $P\in\mathcal{P}(I)$. Then $P$ is a prime $t$-ideal, and hence $I_P$ is a principal ideal, i.e., $I_P=aD_P$ for some $a\in P$. There is some $b\in P$ such that $a\in bD$ and $\sqrt{bD}\in {\rm Spec}(D)$. We have that $P_P=\sqrt{I_P}=\sqrt{aD_P}\subseteq\sqrt{bD_P}\subseteq P_P$. Consequently, $P_P=\sqrt{bD_P}=(\sqrt{bD})_P$, and hence $P=\sqrt{bD}$.

Thus, by \cite[Lemma 2.5]{r12}, $\mathcal{P}(I)$ is finite. Let $\emptyset\not=\mathcal{L}\subseteq\mathcal{P}(I)$. If $Q\in\mathcal{L}$, then $Q=\sqrt{a_QD}$ for some $a_Q\in D$. This implies that $\bigcap_{Q\in\mathcal{L}} Q=\bigcap_{Q\in\mathcal{L}}\sqrt{a_QD}=\sqrt{(\prod_{Q\in\mathcal{L}} a_Q)D}$.
\end{proof}

\begin{remark}\label{remark4.8}
{\em Let $D$ be a VFD in which each minimal prime ideal of a nonzero $t$-finite $t$-ideal is minimal over a $t$-invertible $t$-ideal and let $I$ be a nonzero $t$-finite $t$-ideal of $D$. Then $\mathcal{P}(I)$ is finite and for each $\emptyset\not=\mathcal{L}\subseteq\mathcal{P}(I)$, $\bigcap_{Q\in\mathcal{L}} Q$ is the radical of a principal ideal of $D$.}
\end{remark}

\begin{proof}
First we show that every minimal prime ideal of $I$ is the radical of a principal ideal of $D$. Let $P\in\mathcal{P}(I)$. There is some $t$-invertible $t$-ideal $J$ of $D$ such that $P\in\mathcal{P}(J)$. Consequently, $P$ is the radical of a principal ideal of $D$ by Lemma~\ref{lemma4.7}. Thus, by the proof of Lemma~\ref{lemma4.7}, $\mathcal{P}(I)$ is finite and $\bigcap_{Q\in\mathcal{L}} Q$ is the radical of a principal ideal of $D$ for all $\emptyset\not=\mathcal{L}\subseteq\mathcal{P}(I)$.
\end{proof}

\begin{proposition}\label{proposition4.9}
Let $D$ be a VFD in which each minimal prime ideal of a nonzero $t$-finite $t$-ideal is minimal over a $t$-invertible $t$-ideal. Then every $t$-finite $t$-ideal $I$ of $D$ with $\sqrt{I}\in {\rm Spec}(D)$ is principal.
\end{proposition}

\begin{proof}
It is sufficient to show by induction that for every $m\in\mathbb{N}$ and every finite $E\subseteq D\setminus\{0\}$ such that $\sum_{e\in E} |\mathcal{P}(eD)|=m$ and $\sqrt{(E)_t}\in {\rm Spec}(D)$, it follows that $(E)_t$ is principal.

Let $m\in\mathbb{N}$ and $E\subseteq D\setminus\{0\}$ be such that $E$ is finite, $\sum_{e\in E} |\mathcal{P}(eD)|=m$ and $\sqrt{(E)_t}\in {\rm Spec}(D)$. Set $I=(E)_t$, $Q=\sqrt{I}$ and $\Sigma=\{e\in E\mid Q\in\mathcal{P}(eD)\}$. By Remark~\ref{remark4.8}, there is some $a\in D$ such that $Q=\sqrt{aD}$. Without restriction we can assume that $a\in I$. It follows from Proposition~\ref{proposition1.13} that for each $e\in E$, $e=\prod_{P\in\mathcal{P}(eD)} e_P$, where $e_A$ is a valuation element of $D$ with $\sqrt{e_AD}=A$ for each $A\in\mathcal{P}(eD)$.

\medskip
{\noindent}{\textsc{Case 1}}: $\Sigma=\emptyset$. Let $e\in E$. Since $e\in Q$, there is some $P\in\mathcal{P}(eD)$ such that $P\subsetneq Q$. This implies that $\sqrt{e_PD}\subsetneq\sqrt{aD}$, and hence $e\in e_PD\subseteq aD$ by Corollary~\ref{corollary1.2}(1). Consequently, $E\subseteq aD$, and thus $I=aD$.

\medskip
{\noindent}{\textsc{Case 2}}: $\Sigma\not=\emptyset$. Observe that if $g,h\in\Sigma$, then $\sqrt{g_QD}=\sqrt{h_QD}=Q$, and thus $g_QD$ and $h_QD$ are comparable by Corollary~\ref{corollary1.2}(1). Since $\Sigma$ is finite and nonempty, there is some $f\in\Sigma$ such that $e_QD\subseteq f_QD$ for all $e\in\Sigma$. Set $c=f_Q$.

Next we show that $e\in cD$ and $|\mathcal{P}(ec^{-1}D)|\leq |\mathcal{P}(eD)|$ for each $e\in E$. Let $e\in E$. Then $P\subseteq Q$ for some $P\in\mathcal{P}(eD)$. If $e\in\Sigma$, then $P=Q$ and $e\in e_PD\subseteq cD$. If $e\not\in\Sigma$, then $\sqrt{e_PD}\subsetneq\sqrt{cD}$, and thus $e\in e_PD\subseteq cD$ by Corollary~\ref{corollary1.2}(1). In any case we have that $e_P\in cD$. Set $d=e_Pc^{-1}$. Note that $d\mid_D e_P$ and $e_P$ is a valuation element of $D$. It follows from Proposition~\ref{proposition1.1}(2) that $d$ is either a unit or a valuation element of $D$. Note that $e=e_Pb$, where $b\in D$ is a product of $|\mathcal{P}(eD)|-1$ valuation elements of $D$. Consequently, $ec^{-1}=db$ is a unit or a product of at most $|\mathcal{P}(eD)|$ valuation elements of $D$. It follows from Lemma~\ref{lemma1.12} that $|\mathcal{P}(ec^{-1}D)|\leq |\mathcal{P}(eD)|$.

We infer that $I\subseteq cD$. Set $F=\{ec^{-1}\mid e\in E\}$ and $J=(F)_t$. Clearly, $F\subseteq D\setminus\{0\}$ is finite, $J$ is a $t$-finite $t$-ideal of $D$ and $I=cJ$. Note that $\mathcal{P}(fc^{-1}D)=\mathcal{P}(fD)\setminus\{Q\}$, and thus $|\mathcal{P}(fc^{-1}D)|<|\mathcal{P}(fD)|$. Therefore, $\sum_{g\in F} |\mathcal{P}(gD)|=\sum_{e\in E} |\mathcal{P}(ec^{-1}D)|<\sum_{e\in E} |\mathcal{P}(eD)|=m$. It follows by Remark~\ref{remark4.8} that $\sqrt{J}=\sqrt{bD}$ for some $b\in D$. Without restriction let $J\not=D$. Since $\sqrt{cD}\subseteq\sqrt{bD}$, it follows by Proposition~\ref{proposition1.1}(3) that $b$ is a valuation element of $D$, and hence $\sqrt{J}\in {\rm Spec}(D)$. It follows by the induction hypothesis that $J$ is principal. Consequently, $I$ is principal.
\end{proof}

\begin{proposition}\label{proposition4.10}
Let $D$ be a VFD in which each minimal prime ideal of a nonzero $t$-finite $t$-ideal is minimal over a $t$-invertible $t$-ideal. Then $D$ is a GCD-domain.
\end{proposition}

\begin{proof}
By Remark~\ref{remark4.8} it is sufficient to show by induction that for each $n\in\mathbb{N}$ and every nonzero $t$-finite $t$-ideal $I$ of $D$ with $|\mathcal{P}(I)|=n$, it follows that $I$ is principal.

Let $n\in\mathbb{N}$ and let $I$ be a nonzero $t$-finite $t$-ideal of $D$ such that $|\mathcal{P}(I)|=n$. Without restriction let $n\geq 2$ and let $P\in\mathcal{P}(I)$. By Remark~\ref{remark4.8} there are some $c,d\in D$ such that $P=\sqrt{cD}$ and $\bigcap_{Q\in\mathcal{P}(I)\setminus\{P\}} Q=\sqrt{dD}$. Observe that $\sqrt{I}=\sqrt{cdD}$, and hence $c^kd^k\in I$ for some $k\in\mathbb{N}$. Set $a=c^k$ and $b=d^k$. Then $P=\sqrt{aD}$, $\bigcap_{Q\in\mathcal{P}(I)\setminus\{P\}} Q=\sqrt{bD}$ and $ab\in I$. Set $J=(I+aD)_t$. Then $J$ is a $t$-finite $t$-ideal of $D$ such that $\sqrt{J}=P$, and hence $J$ is principal by Proposition~\ref{proposition4.9}. Consequently, there is some $t$-finite $t$-ideal $L$ of $D$ such that $I=JL$.

Next we show that $\mathcal{P}(L)=\mathcal{P}(I)\setminus\{P\}$. First let $A\in\mathcal{P}(I)\setminus\{P\}$. Then $JL=I\subseteq A$. If $J\subseteq A$, then $P\subseteq A$, and hence $P=A$, a contradiction. Therefore, $L\subseteq A$, and since $I\subseteq L$, we infer that $A\in\mathcal{P}(L)$. Now let $B\in\mathcal{P}(L)$. Since $ab\in I$, we have that $Jb\subseteq I$, and hence $b\in L\subseteq B$. Consequently, $\sqrt{bD}\subseteq B$. This implies that $C\subseteq B$ for some $C\in\mathcal{P}(I)\setminus\{P\}$. We have that $C\in\mathcal{P}(L)$ (as shown before), and thus $B=C\in\mathcal{P}(I)\setminus\{P\}$. By the induction hypothesis, $L$ is principal. Thus, $I=JL$ is principal.
\end{proof}

We do not know whether every VFD is a weakly Matlis GCD-domain, but we do know this is affirmative under certain additional assumptions. In what follows, we summarize a variety of conditions that force a VFD to be a weakly Matlis GCD-domain.

\begin{theorem}\label{theorem4.11}
The following statements are equivalent for a VFD $D$.
\begin{enumerate}
\item $D$ is a weakly Matlis GCD-domain.
\item $D$ is a UVFD.
\item $D$ is a P$v$MD.
\item $D$ is a $t$-treed domain.
\item Each minimal prime ideal of each nonzero $t$-finite $t$-ideal of $D$ is minimal over a $t$-invertible $t$-ideal.
\item $D$ is a $t$-finite conductor domain.
\item $D$ is a UMT-domain.
\end{enumerate}
\end{theorem}

\begin{proof}
(1) $\Leftrightarrow$ (2) This is an immediate consequence of Corollary~\ref{corollary4.5}.

(2) $\Leftrightarrow$ (3) This follows from Theorem~\ref{theorem4.2}.

(3) $\Rightarrow$ (4), (5), (6), (7) This is clear.

(4) $\Rightarrow$ (3) This follows from Theorem~\ref{theorem3.4}.

(5) $\Rightarrow$ (3) This is an immediate consequence of Proposition~\ref{proposition4.10}.

(6) $\Rightarrow$ (3) This follows from Corollary~\ref{corollary2.3}.

(7) $\Rightarrow$ (3) Note that every integrally closed UMT-domain is a P$v$MD, and thus the statement follows by Corollary~\ref{corollary1.5}.
\end{proof}

\begin{proposition}\label{proposition4.12}
Let $D$ be a VFD and let $\Omega=\{\sqrt{xD}\mid x\in D\setminus\{0\},\sqrt{xD}\in {\rm Spec}(D)\}$. Let one of the following conditions be satisfied.
\begin{enumerate}
\item For all $P\in\Omega$, each nonzero prime $t$-ideal of $D$ contained in $P$ is in $\Omega$.
\item For all $P\in\Omega$ there is a unique height-one prime ideal $Q$ of $D$ with $Q\subseteq P$.
\item $t\textnormal{-}\dim(D)\leq 2$.
\end{enumerate}
Then $D$ is a weakly Matlis GCD-domain.
\end{proposition}

\begin{proof}
(1) By Theorems~\ref{theorem4.2} and~\ref{theorem4.11} it remains to show that $D_P$ is a valuation domain for each $P\in\Omega$. Let $P\in\Omega$. Then $D_P$ is a VFD by Corollary~\ref{corollary1.8}(1). Moreover, $P_P$ is both the unique maximal $t$-ideal of $D_P$ and the radical of a principal ideal of $D_P$. Next we show that every nonzero prime $t$-ideal of $D_P$ is the radical of a principal ideal of $D_P$. Let $Q$ be a nonzero prime $t$-ideal of $D_P$. Then $Q\cap D$ is a nonzero prime $t$-ideal of $D$ contained in $P$. Therefore, $Q\cap D$ is the radical of a principal ideal of $D$, and thus $Q=(Q\cap D)_P$ is the radical of a principal ideal of $D_P$. Consequently, $D_P$ satisfies (5) in Theorem~\ref{theorem4.11}, and thus $D_P$ is a P$v$MD again by Theorem~\ref{theorem4.11}. We infer that $D_P$ is a valuation domain (since $P_P$ is a maximal $t$-ideal of $D_P$).

(2) By Theorems~\ref{theorem4.2} and~\ref{theorem4.11} it is sufficient to show that for all $A,B,C\in\Omega$ with $A\cup B\subseteq C$, $A$ and $B$ are comparable. Let $A,B,C\in\Omega$ be such that $A\cup B\subseteq C$. There are some height-one prime ideals $P$ and $Q$ of $D$ such that $P\subseteq A$ and $Q\subseteq B$. Since $P\cup Q\subseteq C$, it follows that $P=Q$, and hence $P\subseteq A\cap B$. Therefore, $A$ and $B$ are comparable by Corollary~\ref{corollary1.14}(2).

(3) Note that $\Omega$ contains the set of height-one prime ideals of $D$, and thus $D$ is a weakly Matlis GCD-domain by (1).
\end{proof}

A weakly Matlis GCD-domain is a VFD by Corollary~\ref{corollary4.5}. Moreover, if $D$ is a $t$-treed domain (e.g., P$v$MD or $t$-dim$(D)=1$), then $D$ is a VFD if and only if $D$ is a weakly Matlis GCD-domain by Theorem~\ref{theorem3.4}. We end this paper with a question.

\begin{question}\label{question4.13}
{\em Let $D$ be a VFD. Is $D$ a weakly Matlis GCD-domain?}
\end{question}

\bigskip
\noindent
\textbf{Acknowledgements.} We want to thank the referee for many helpful suggestions and comments which improved the quality of this paper. This research was completed while the second-named author visited Incheon National University during 2019. The first-named author was supported by Basic Science Research Program through the National Research Foundation of Korea (NRF) funded by the Ministry of Education (2017R1D1A1B06029867). The second-named author was supported by the Austrian Science Fund FWF, Project Number J4023-N35.

\end{document}